\documentclass[12pt]{amsart}
\usepackage{amscd,amsmath,amsthm,amssymb}
\usepackage{pstcol,pst-plot,pst-3d}
\usepackage{lmodern,pst-node}
\usepackage{multicol}
\usepackage{epic,eepic}
\usepackage{amsfonts,amssymb,amscd,amsmath,enumerate,verbatim}
\psset{unit=0.7cm,linewidth=0.8pt,arrowsize=2.5pt 4}

\newpsstyle{fatline}{linewidth=1.5pt}
\newpsstyle{fyp}{fillstyle=solid,fillcolor=verylight}
\definecolor{verylight}{gray}{0.97}
\definecolor{light}{gray}{0.9}
\definecolor{medium}{gray}{0.85}
\definecolor{dark}{gray}{0.6}

\def\NZQ{\Bbb}               
\def\NN{{\NZQ N}}

\def\ZZ{{\NZQ Z}}

\def\frk{\frak}               

\def\Phi{{\frk n}}
\def\Phi{{\frk N}}
\def\Lc{{\mathcal L}}
\def\Cc{{\mathcal C}}
\def\Gc{{\mathcal G}}

\def\Ac{{\mathcal A}}
\def\Pc{{\mathcal P}}
\def\Bc{{\mathcal B}}

\def\Sc{{\mathcal S}}
\def\ab{{\bold a}}
\def\bb{{\bold b}}
\def\xb{{\bold x}}

\def\vb{{\bold v}}
\def\opn#1#2{\def#1{\operatorname{#2}}} 
\opn\chara{char} \opn\length{\ell} \opn\pd{pd} \opn\rk{rk}
\opn\projdim{proj\,dim} \opn\injdim{inj\,dim} \opn\rank{rank}
\opn\depth{depth} \opn\grade{grade} \opn\height{height}
\opn\embdim{emb\,dim} \opn\codim{codim}

\opn\Tr{Tr} \opn\bigrank{big\,rank}
\opn\superheight{superheight}\opn\lcm{lcm}
\opn\trdeg{tr\,deg}
\opn\reg{reg} \opn\lreg{lreg} \opn\ini{in} \opn\lpd{lpd}
\opn\size{size} \opn\sdepth{sdepth}
\opn\link{link}\opn\fdepth{fdepth}\opn\lex{lex}
\opn\div{div} \opn\Div{Div} \opn\cl{cl} \opn\Cl{Cl}
\opn\Spec{Spec} \opn\Supp{Supp} \opn\supp{supp} \opn\Sing{Sing}
\opn\Ass{Ass} \opn\Min{Min}\opn\Mon{Mon}
\opn\Ann{Ann} \opn\Rad{Rad} \opn\Soc{Soc}
\opn\Im{Im} \opn\Ker{Ker} \opn\Coker{Coker} \opn\Am{Am}
\opn\Hom{Hom} \opn\Tor{Tor} \opn\Ext{Ext} \opn\End{End}
\opn\Aut{Aut} \opn\id{id}

\opn\nat{nat}
\opn\pff{pf}
\opn\Pf{Pf} \opn\GL{GL} \opn\SL{SL} \opn\mod{mod} \opn\ord{ord}
\opn\Gin{Gin} \opn\Hilb{Hilb}\opn\sort{sort}
\opn\aff{aff} \opn\con{conv} \opn\relint{relint} \opn\st{st}
\opn\lk{lk} \opn\cn{cn} \opn\core{core} \opn\vol{vol}
\opn\link{link} \opn\star{star}\opn\lex{lex}\opn\set{set}
\opn\gr{gr}

\def\pot#1#2{#1[\kern-0.28ex[#2]\kern-0.28ex]}

\opn\dirlim{\underrightarrow{\lim}}
\opn\inivlim{\underleftarrow{\lim}}
\let\union=\cup
\let\sect=\cap

\let\Union=\bigcup
\let\Sect=\bigcap

\def\Implies{\ifmmode\Longrightarrow \else
        \unskip${}\Longrightarrow{}$\ignorespaces\fi}
\def\implies{\ifmmode\Rightarrow \else
        \unskip${}\Rightarrow{}$\ignorespaces\fi}
\def\iff{\ifmmode\Longleftrightarrow \else
        \unskip${}\Longleftrightarrow{}$\ignorespaces\fi}

\let\:=\colon
\newtheorem{Theorem}{Theorem}[section]
\newtheorem{Lemma}[Theorem]{Lemma}
\newtheorem{Corollary}[Theorem]{Corollary}
\newtheorem{Proposition}[Theorem]{Proposition}

\newtheorem{Example}[Theorem]{Example}

\let\epsilon\varepsilon
\let\kappa=\varkappa
\def\qed{\ifhmode\textqed\fi
      \ifmmode\ifinner\quad\qedsymbol\else\dispqed\fi\fi}
\def\textqed{\unskip\nobreak\penalty50
       \hskip2em\hbox{}\nobreak\hfil\qedsymbol
       \parfillskip=0pt \finalhyphendemerits=0}
\def\dispqed{\rlap{\qquad\qedsymbol}}

\opn\dis{dis}
\def\pnt{{\raise0.5mm\hbox{\large\bf.}}}

\opn\Lex{Lex}

\begin{document}

\title {Ideals generated by adjacent $2$-minors}

\author {J\"urgen Herzog and Takayuki Hibi}


\address{J\"urgen Herzog, Fachbereich Mathematik, Universit\"at Duisburg--Essen, Campus Essen, 45117
Essen, Germany} \email{juergen.herzog@uni-essen.de}

\address{Takayuki Hibi, Department of Pure and Applied Mathematics, Graduate School of Information Science and Technology,
Osaka University, Toyonaka, Osaka 560-0043, Japan}
\email{hibi@math.sci.osaka-u.ac.jp}

\begin{abstract}
Ideals generated by adjacent $2$-minors are studied.
First, the problem when such an ideal is a prime ideal
as well as the problem when such an ideal possesses a quadratic Gr\"obner basis
is solved.
Second, we describe explicitly
a primary decomposition of the radical ideal of
an ideal generated by adjacent $2$-minors,
and challenge the question of classifying all ideals generated by adjacent $2$-minors
which are radical ideals.
Finally, we discuss connectedness of contingency tables in algebraic statistics.

\end{abstract}
\maketitle

\section*{Introduction}
Let $X=(x_{ij})_{i=1,\ldots,m\atop j=1,\ldots,m}$ be an $m\times n$-matrix of indeterminates, and let $K$ be an arbitrary field. The ideals of $t$-minors $I_t(X)$  in $K[X]=K[(x_{ij})_{i=1,\ldots,m\atop j=1,\ldots,m}]$ are well understood. A standard reference for determinantal ideals are the lecture notes \cite{BV} by Bruns and Vetter.  See also \cite{BH} for a short introduction to this subject.  Determinantal ideals  and the natural extensions of this class of ideals, including ladder determinantal ideals arise naturally in geometric contexts which partially explains the interest in them. One nice property of these ideals is that they are all Cohen--Macaulay prime ideals.

Motivated by applications to algebraic statistics one is lead to study ideals generated by an {\em arbitrary} set of $2$-minors of $X$. We refer the interested reader to the article \cite{DES} of Diaconis,  Eisenbud and Sturmfels where the encoding of the statistical problem to commutative algebra is nicely described. Here we just outline briefly some ideas presented in that paper. Let $\Bc$ be a subset of vectors of $\ZZ^n$. One defines the graph $G_\Bc$ whose vertex set is the set $\NN^n$ of  non-negative integer vectors. Two vectors $\ab$ and $\bb$ are connected by an edge of  $G_\Bc$  if $\ab-\bb\in\pm \Bc$. One of the central problems is to describe the connected components of this graph.  Here commutative algebra comes into play. One defines the ideal
\[
I_\Bc=(\xb^{\bb_+}-\xb^{\bb_-}\:\, \bb\in \Bc),
\]
where for a vector $\ab\in \ZZ^n$, the vectors $\ab_+, \ab_-\in \NN^n$  are the unique vectors with $\ab=\ab_+-\ab_-$. The crucial observation  is the following  \cite[Theorem 1.1.]{DES}: two vectors $\ab,\bb\in \NN^n$ belong to the same component of $G_\Bc$ if and only if $\xb^\ab-\xb^\bb\in I_\Bc$. Hence the question arises how one can decide whether a specific binomial $f=\xb^\ab-\xb^\bb$ belongs to a given binomial ideal $I$. This is indeed a difficult problem, but often  the following strategy yields necessary  and sometimes necessary and sufficient conditions for $f$ to belong to $I$, which can be expressed in terms of feasible  numerical conditions on the vectors $\ab$ and $\bb$. The idea is to write the given binomial ideal as an intersection $I=\Sect_{k=1}^rJ_k$ of ideals $J_k$. Then $f\in I$ if and only if $f\in J_k$ for all $k$. This strategy is useful only if each of the ideals $J_k$ has a simple structure, so that it is possible to describe the conditions that guarantee that $f$ belongs to $J_k$. A natural choice for such an intersection is a primary decomposition of $I$. By Eisenbud and Sturmfels \cite{ES} it is known that the primary components of $I$ are again binomial ideals (in the more general sense that a binomial ideal may also contain monomials). In the  case that $I$ is a radical ideal the natural choice for the ideals $J_k$ are the minimal prime ideals of $I$.

To be more specific we consider contingency tables with support in a subset
$\Sc\subset [m]\times [n]$,
where $\Sc$ is the set of certain
unit boxes
$$\{(i,j), (i,j+1), (i+1,j),(i+1,j+1)\}$$
with $1\leq i<m$ and $1\leq j<n$.
We call a set $T=\{a_{ij}\:\; (i,j)\in \Sc \}$ of non-negative integers a contingency table with support in $\Sc$. From the statistical point of view it is of interest to generate random contingency tables with fixed row and column sum. The row sums of the table $T$ are the sums $\sum_{j, \; (i,j)\in \Sc}a_{ij}$ and the column sums of $T$ are the sums  $\sum_{i,\; (i,j)\in \Sc}a_{ij}$.

Choosing $i_1<i_2$ and $j_1<j_2$ such that
$$Q=\{(i_1,j_1),(i_1,j_2),(i_2,j_1),(i_1,j_2)\}\subset \Sc,$$
we may produce a new contingency table $T'$ with same row and column sums by adding  or subtracting the table $D(Q)=\{d_{ij}\:\; (i,j)\in \Sc\}$, where $d_{ij}=0$ if $(i,j)\notin Q$, $d_{i_1,j_1}=d_{i_2,j_2}=1$ and $d_{i_1,j_2}=d_{i_2,j_1}=-1$, provided all entries of $T'$ are again non-negative.

Given a set $\mathcal{D}=\{D_1,\ldots,D_r\}$ of tables as described in the preceding paragraph, we may start a random walk beginning with $T$ by adding or subtracting in each step a randomly chosen table $D_i$. If in a step like this, which we call an adjacent move,   the new table has non-negative entries we continue our walk, otherwise we go back one step and choose another $D_i$.  The contingency tables which we can reach by this procedure from the given table $T$ are called {\em connected} to $T$ with respect to $\mathcal{D}$.

In this paper we are interested in walks where $\mathcal{D}$ consists of the  tables $D(Q)$,  where $Q$ runs though all unit boxes in $\Sc$. By the above mentioned \cite[Theorem 1.1]{DES} we are lead to study ideals generated by adjacent $2$-minors, that is, minors of the form $x_{i,j}x_{i+1,j+1}-x_{i+1,j}x_{i,j+1}$. Following the general strategy described above it is desirable to understand the primary decomposition of such ideals, or at least their minimal prime ideals. Ho\c{s}ten and Sullivant \cite{HS2} describe in a very explicit way all the minimal prime ideals for the ideal generated by all adjacent $2$-minors of an $m\times n$-matrix. In  Theorem~\ref{minimalprimes} we succeed in describing the minimal prime ideals of the ideal of a configuration of adjacent $2$-minors under the mild assumption that this configuration is special, a concept which has first been introduced by Qureshi \cite{Q}. Special configurations include the case considered by Ho\c{s}ten and Sullivant, but are much more general. Our description is not quite as explicit as that  of Ho\c{s}ten and Sullivant, but explicit enough to determine in each particular case all the minimal prime ideals. Part of the result given in Theorem~\ref{minimalprimes} can also be derived from \cite[Corollary2.1]{HS1} of Ho\c{s}ten and Shapiro, since ideals generated by $2$-adjacent minors are lattice basis ideals. Though the minimal prime ideals are known, the knowledge about the embedded prime ideals of an ideal generated by 2-adjacent minors is very little, let alone the knowledge on its primary decomposition. In \cite{DES} the primary decomposition of the ideal of all adjacent $2$-minors of a $4\times 4$-matrix is given, and in \cite{HS1} that of a $3\times 5$-matrix. It is hard to see a general pattern from these results.

Ideals generated by adjacent $2$-minors tend to have a  non-trivial radical, and are rarely prime ideals. In the first section of this paper we classify all ideals generated by adjacent $2$-minors which are prime ideals. The result is described in Theorem~\ref{prime}. They are the ideals of adjacent $2$-minors attached to a chessboard configuration with no 4-cycles.

One method to show that an ideal is a radical ideal, is to compute its initial ideal with respect to some monomial order. If the initial ideal is squarefree, then the given ideal is a radical ideal. In Section 2 we classify all ideals generated by adjacent $2$-minors which have a quadratic Gr\"obner basis (Theorem~\ref{quadratic}). It turns out that these are the ideals of adjacent $2$-minors corresponding to configurations whose components are  monotone  paths meeting is a suitable way. In particular, those ideals of adjacent $2$-minors are radical ideals. In general the radical of an ideal of adjacent $2$-minors attached to a special configuration can be naturally written as an intersections of prime ideals of relatively simple nature, see Theorem~\ref{primeintersection}. These prime ideals are indexed by the so-called admissible sets. These are subsets of the set $\Sc$ which defines the configuration, and can be described in a purely combinatorial way.

In Section 4 we aim at classifying configurations whose ideal  of adjacent $2$-minors is a radical ideal. It is not so hard to see (cf.\ Proposition~\ref{onedirection}) that a connected special  configuration whose ideal  of adjacent $2$-minors is a radical ideal should be a path or a cycle.  Computations show that the cycles should have  length at least 12. We expect that the ideal of adjacent $2$-minors attached to any path is a radical ideal, and prove this in Theorem~\ref{radicalpath} under the additional assumption that the ideal has no embedded prime ideals.

In the last section of this paper we describe in detail what it means that special contingency tables are connected via adjacent moves.

\section{Prime ideals generated by adjacent $2$-minors}

Let $X=(x_{ij})_{i=1,\ldots,m\atop j=1,\ldots,n}$ be a matrix of indeterminates, and let $S$ be the polynomial ring over a field $K$ in the variables $x_{ij}$. Let  $\delta =[a_1,a_2|b_1,b_2]$ be a $2$-minor. The variables  $x_{a_i,b_j}$ are called the {\em vertices} and the sets $\{x_{a_1,b_1}, x_{a_1,b_2}\}]$, $\{x_{a_1,b_1},x_{a_2,b_1}\}$, $\{x_{a_1,b_2},x_{a_2,b_2}\}$ and $\{x_{a_2,b_1},x_{a_2,b_2}\}$  the {\em edges} of the minor $[a_1,a_2|b_1,b_2]$. The set of vertices of $\delta$ will be denoted by $V(\delta)$. The $2$-minor  $\delta =[a_1,a_2|b_1,b_2]$ is called {\em adjacent} if $a_2=a_1+1$ and $b_2=b_1+1$.

Let $\Cc$ be any set of adjacent $2$-minors. We call such a set also a configuration of adjacent $2$-minors.  We denote by $I(\Cc)$ the  ideal generated by the elements of $\Cc$.
The set of vertices   of $\Cc$,  denoted $V(\Cc)$, is the union of the vertices of its adjacent $2$-minors. Two distinct minors in $\delta,\gamma\in \Cc$ are called {\em connected} if there exist $\delta_1\ldots,\delta_r\in \Cc$ such that $\delta=\delta_1$, $\gamma=\delta_r$,  and $\delta_i$  and $\delta_{i+1}$ have a common edge.

Following Quereshi \cite{Q} we call a configuration of adjacent $2$-minors {\em special}, if the following condition is satisfied: for any two adjacent $2$-minors $\delta_1, \delta_2\in \Cc$ which have exactly one vertex in common, there exists a $\delta\in \Cc$ which has a common edge with $\delta_1$ and a common edge with $\delta_2$. Any special configuration of adjacent $2$-minors is a disjoint union of connected special configuration of adjacent $2$-minors.

Any maximal subset $D$ of $\Cc$ with the property that any two minors of $D$ are connected, is called a {\em connected component of $\Cc$}. To $\Cc$ we attach a graph $G_\Cc$ as follows: the vertices of $G_\Cc$ are the connected components  of $\Cc$. Let $A$ and $B$ be two connected components of $C$. Then there is an edge between  $A$ and $B$ if there exists a minor $\delta\in A$ and a minor $\gamma\in B$ which have exactly one vertex in common. Note that $G_\Cc$ may have multiple edges.

A set of adjacent $2$-minors is called a {\em chessboard configuration}, if any two  minors of this set meet in at most one vertex. An example of a chessboard configuration is given in Figure~\ref{chessboard}.  An ideal $I\subset S$ is called a {\em chessboard ideal} if it is generated by a chessboard configuration.  Note that the graph $G_\Cc$
of a chessboard configuration is a simple bipartite graph.

\begin{figure}[hbt]
\begin{center}
\psset{unit=0.9cm}
\begin{pspicture}(4.5,-0.5)(4.5,4)
\pspolygon[style=fyp,fillcolor=light](3,0)(3,1)(4,1)(4,0)
\pspolygon[style=fyp,fillcolor=light](2,1)(2,2)(3,2)(3,1)
\pspolygon[style=fyp,fillcolor=light](4,1)(4,2)(5,2)(5,1)
\pspolygon[style=fyp,fillcolor=light](3,2)(3,3)(4,3)(4,2)
\pspolygon[style=fyp,fillcolor=light](3,0)(3,1)(4,1)(4,0)
\pspolygon[style=fyp,fillcolor=light](5,0)(5,1)(6,1)(6,0)
\pspolygon[style=fyp,fillcolor=light](5,2)(5,3)(6,3)(6,2)
\pspolygon[style=fyp,fillcolor=light](6,3)(6,4)(7,4)(7,3)
\end{pspicture}
\end{center}
\caption{}\label{chessboard}
\end{figure}

\begin{Theorem}
\label{prime}
Let $I$ be an ideal generated by adjacent $2$-minors. Then the following conditions are equivalent:
\begin{enumerate}
\item[{\em (a)}] $I$ is a prime ideal.
\item[{\em (b)}] $I$ is a chessboard ideal and $G_\Cc$ has no cycle of length $4$.
\end{enumerate}
\end{Theorem}

For the proof of this result we shall need some concepts related to lattice ideals.

Let $\Lc\subset \ZZ^n$ be a  lattice. Let $K$ be a field. The {\em lattice ideal}   attached to $\Lc$ is the binomial ideal $I_\Lc\subset K[x_1,\ldots,x_n]$ generated by all binomials
\[
\xb^\ab-\xb^\bb\quad \text{with}\quad \ab-\bb\in \Lc.
\]
$\Lc$ is called saturated if for all $\ab\in \ZZ^n$ and $c\in \ZZ$ such that $c\ab\in \Lc$ it follows that $\ab\in \Lc$. The lattice ideal $I_\Lc$ is a prime ideal if and only if $\Lc$ is saturated.

Let $\vb_1,\ldots, \vb_m$ be a basis of $\Lc$. Ho\c{s}ten and Shapiro \cite{HS1} call the ideal generated by the binomials $\xb^{\vb_i^+}-\xb^{\vb_i^-}$, $i=1,\ldots,m$, a {\em lattice basis ideal} of $\Lc$. Here $\vb^+$ denotes the vector obtained from $\vb$ by replacing all negative components of $\vb$ by zero, and $\vb^-=-(\vb-\vb^+)$.

Fischer and Shapiro \cite{FS}, and Eisenbud and Sturmfels \cite{ES} showed

\begin{Proposition}
\label{saturation}
Let $J$ be a lattice basis ideal of the saturated lattice $\Lc\subset \ZZ^n$. Then $J\: (\prod_{i=1}^nx_i)^\infty =I_\Lc$.
\end{Proposition}

As a consequence of this proposition we obtain:

\begin{Lemma}
\label{regular}
Let $I$ be an ideal generated by adjacent $2$-minors. Then $I$ is a prime ideal if and only if all variables $x_{ij}$ are nonzero divisors of $S/I$.
\end{Lemma}

\begin{proof}
It is known from \cite{ES} that the ideal generated by all adjacent $2$-minors of $X$ is a lattice basis ideal, and that the corresponding lattice ideal is just the ideal of all $2$-minors of $X$. It follows that an ideal, such as $I$, which is  generated by any set of adjacent $2$-minors of $X$ is again a lattice basis ideal and that its corresponding lattice $\Lc$ is saturated. Therefore  its lattice ideal  $I_\Lc$ is a prime ideal. Assume now that all variables $x_{ij}$ are nonzero divisors of $S/I$. Then Proposition~\ref{saturation} implies that $I=I_\Lc$, so that $I$ is a prime ideal. The converse implication is trivial.
\end{proof}

For the proof of Theorem~\ref{prime} we need the following two lemmata.

\begin{Lemma}
\label{choice}
Let $I$ be an ideal generated by adjacent  $2$-minors.  For  each of the minors we mark one of the monomials in the support as a potential initial monomial. Then there exists an ordering  of the variables such that the marked monomials are indeed the initial monomials with respect to the lexicographic order induced by the given ordering of the variables.
\end{Lemma}

\begin{proof}
In general, suppose that, in the set $[N] = \{1, 2, \ldots, N\}$,
for each pair $(i, i+1)$ an ordering either $i < i + 1$ or $i > i + 1$ is given.
We claim that there is a total order $<$ on $[N]$ which preserves the given ordering.
Working by induction on $N$ we may assume that
there is a total order $i_1 < \ldots < i_{N-1}$ on $[N-1]$ which preserve the given ordering
for the pairs $(1, 2), \ldots, (N-2, N-1)$.
If $N-1 < N$, then $i_1 < \ldots < i_{N-1} < N$ is a required total order $<$ on $[N]$.
If $N-1 > N$ , then $N < i_1 < \ldots < i_{N-1}$ is a required total order $<$ on $[N]$.

The above fact guarantees the existence
of an ordering of the variables such that the marked monomials are
indeed the initial monomials with respect to the lexicographic order induced
by the given ordering of the variables, as can be seen in Example~\ref{order}.
\end{proof}

\begin{figure}[hbt]
\begin{center}
\psset{unit=0.9cm}
\begin{pspicture}(4.5,-0.5)(4.5,4)
\pspolygon[style=fyp,fillcolor=light](4,0)(4,1)(5,1)(5,0)
\pspolygon[style=fyp,fillcolor=light](3,2)(3,3)(4,3)(4,2)
\pspolygon[style=fyp,fillcolor=light](5,1)(5,2)(6,2)(6,1)
\pspolygon[style=fyp,fillcolor=light](4,2)(4,3)(5,3)(5,2)
\pspolygon[style=fyp,fillcolor=light](4,0)(4,1)(5,1)(5,0)
\pspolygon[style=fyp,fillcolor=light](2,2)(2,3)(3,3)(3,2)
\pspolygon[style=fyp,fillcolor=light](5,2)(5,3)(6,3)(6,2)
\pspolygon[style=fyp,fillcolor=light](6,0)(6,1)(7,1)(7,0)
\pspolygon[style=fyp,fillcolor=light](6,2)(6,3)(7,3)(7,2)
\pspolygon[style=fyp,fillcolor=light](4,1)(4,2)(5,2)(5,1)
\psline(3,2)(2,3)
\psline(3,2)(4,3)
\psline(4,3)(5,2)
\psline(5,2)(6,3)
\psline(6,2)(7,3)
\psline(5,1)(4,2)
\psline(5,1)(6,2)
\psline(5,0)(4,1)
\psline(6,0)(7,1)
\rput(2,3.3){1}
\rput(3,3.3){2}
\rput(4,3.3){3}
\rput(5,3.3){4}
\rput(6,3.3){5}
\rput(7,3.3){6}
\rput(2,1.7){7}
\rput(3, 1.7){8}
\rput(7, 1.7){12}
\rput(3.8,1.7){9}
\rput(4.7,1.7){10}
\rput(6.25,1.7){11}
\rput(3.65, 0.7){13}
\rput(4.7,0.7){14}
\rput(5.7,0.7){15}
\rput(7.3,0.7){16}
\rput(4,-0.3){17}
\rput(5,-0.3){18}
\rput(6,-0.3){19}
\rput(7,-0.3){20}
\end{pspicture}
\end{center}
\caption{}\label{marked}
\end{figure}

\begin{figure}[hbt]
\begin{center}
\psset{unit=0.9cm}
\begin{pspicture}(4.5,-0.5)(4.5,4)
\pspolygon[style=fyp,fillcolor=light](4,0)(4,1)(5,1)(5,0)
\pspolygon[style=fyp,fillcolor=light](3,2)(3,3)(4,3)(4,2)
\pspolygon[style=fyp,fillcolor=light](5,1)(5,2)(6,2)(6,1)
\pspolygon[style=fyp,fillcolor=light](4,2)(4,3)(5,3)(5,2)
\pspolygon[style=fyp,fillcolor=light](4,0)(4,1)(5,1)(5,0)
\pspolygon[style=fyp,fillcolor=light](2,2)(2,3)(3,3)(3,2)
\pspolygon[style=fyp,fillcolor=light](5,2)(5,3)(6,3)(6,2)
\pspolygon[style=fyp,fillcolor=light](6,0)(6,1)(7,1)(7,0)
\pspolygon[style=fyp,fillcolor=light](6,2)(6,3)(7,3)(7,2)
\pspolygon[style=fyp,fillcolor=light](4,1)(4,2)(5,2)(5,1)
\psline(3,2)(2,3)
\psline(3,2)(4,3)
\psline(4,3)(5,2)
\psline(5,2)(6,3)
\psline(6,2)(7,3)
\psline(5,1)(4,2)
\psline(5,1)(6,2)
\psline(5,0)(4,1)
\psline(6,0)(7,1)
\rput(2,3.3){4}
\rput(3,3.3){5}
\rput(4,3.3){3}
\rput(5,3.3){6}
\rput(6,3.3){2}
\rput(7,3.3){1}
\rput(2,1.7){7}
\rput(3, 1.7){8}
\rput(7, 1.7){12}
\rput(3.8,1.7){10}
\rput(4.7,1.7){11}
\rput(6.25,1.7){9}
\rput(3.65, 0.7){13}
\rput(4.7,0.7){14}
\rput(5.7,0.7){16}
\rput(7.3,0.7){15}
\rput(4,-0.3){17}
\rput(5,-0.3){18}
\rput(6,-0.3){19}
\rput(7,-0.3){20}
\end{pspicture}
\end{center}
\caption{}\label{relabeled}
\end{figure}

\begin{Example}
\label{order}
{\em
In Figure~\ref{marked} each of the squares represents an adjacent $2$-minor, and the diagonal in each of the squares indicates the marked  monomial of the corresponding $2$-minor. For any lexicographic order for which the marked monomials in Figure~\ref{marked} are the initial monomials the numbering of the variables in the top row must satisfy the following inequalities
\[
1<2>3<4>5>6.
\]
By using the general strategy given in the proof of Lemma~\ref{choice} we relabel the top row of the vertices by the numbers $1$ up to $6$, and proceed in the same way in the next rows. The final result can be seen in Figure~\ref{relabeled}
}
\end{Example}

We call a vertex of a $2$-minor in $\Cc$ {\em free}, if it does not belong to any other $2$-minor of $\Cc$,  and we call the  $2$-minor $\delta=ad-bc$ {\em free}, if  either (i) $a$ and $d$ are free, or (ii) $b$ and $c$ are free.

\begin{Lemma}
\label{notallispossible}
Let $\Cc$ be a chessboard configuration with $|\Cc|\geq 2$. Suppose  $G_{\Cc}$  does not contain a cycle of length $4$. Then the $G_{\Cc}$ contains at least two  free $2$-minors.
\end{Lemma}

\begin{proof}
We may assume there is  at least one non-free  $2$-minor in $\Cc$, say  $\delta= ad-bc$.  Since we do not have a cycle of length $4$, there exists a sequence of $2$-minors in $\Cc$ as indicated in Figure~\ref{notall}.  Then the left-most and the right-most 2-minor of this sequence is free.
\end{proof}

\begin{figure}[hbt]
\begin{center}
\psset{unit=0.7cm}
\begin{pspicture}(4.5,0)(4.5,3)
\pspolygon[style=fyp,fillcolor=light](4,0)(4,1)(5,1)(5,0)
\pspolygon[style=fyp,fillcolor=light](3,1)(3,2)(4,2)(4,1)
\pspolygon[style=fyp,fillcolor=light](5,1)(5,2)(6,2)(6,1)
\pspolygon[style=fyp,fillcolor=light](6,0)(6,1)(7,1)(7,0)
\pspolygon[style=fyp,fillcolor=light](7,1)(7,2)(8,2)(8,1)
\pspolygon[style=fyp,fillcolor=light](2,0)(3,0)(3,1)(2,1)
\pspolygon[style=fyp,fillcolor=light](1,1)(2,1)(2,2)(1,2)
\rput(3.7,0.7){a}
\rput(3.8,-0.2){c}
\rput(5.3,0.7){b}
\rput(5.2,-0.2){d}
\rput(8.5,1){$\cdots$}
\rput(9.3,1){$\cdots$}
\rput(0.5,1){$\cdots$}
\rput(-0.3,1){$\cdots$}
\end{pspicture}
\end{center}
\caption{}\label{notall}
\end{figure}

\begin{proof}[Proof of Theorem \ref{prime}]
(a)$\Rightarrow$ (b): Let $\delta,\gamma\in I$ be two adjacent $2$-minors which have an edge in common. Say,  $\delta=ae-bd$ and $\gamma=bf-ce$. Then $b(af-cd)\in I$, but neither $b$ nor $af-cd$ belongs to $I$. Therefore $I$ must be a chessboard ideal. Suppose $G_\Cc$ contains a cycle of length $4$.  Then there exist in $I$ adjacent two minors  $\delta_1=ae-bd$,  $\delta_2=ej-fi$, $\delta_3=hl-ik$ and $\delta_4=ch-dg$. Then $h(bcjk - afgl)\in I$, but neither $h$ nor $bcjk - afgl$ belongs to $I$.

(b)$\Rightarrow$ (a):  By virtue of Lemma \ref{regular} what we must prove is that all variables $x_{ij}$ are nonzero divisors of $S/I$. Let $\Gc$ be the set of generating adjacent $2$-minors of $I$. We may assume that $|\Gc|\geq 2$.
Fix an arbitrary vertex $x_{ij}$.
We claim that for each of the minors in $\Gc$ we may mark one of the monomials in the support as a potential initial monomial such that
the variable $x_{ij}$ appears in none of the potential initial monomials and that any two potential initial monomials are relatively prime.

Lemma \ref{notallispossible} says that there exist at least two a free adjacent $2$-minors in $\Gc$. Let  $\delta=ad-bc$ be one of them and assume that  $a$ and $d$ are free vertices of $\delta$.
We may assume that $x_{ij} \neq a$ and $x_{ij} \neq d$.  Let ${\Gc}^\prime = \Gc \setminus \{ \delta \}$.  By assumption of induction,
for each of the minors of ${\Gc}^\prime$ we may mark one of the monomials in the support as a potential initial monomial such that
the variable $x_{ij}$ appears in none of the potential initial monomials and that any two potential initial monomials are relatively prime.
Then these markings together with the marking $ad$ are the desired markings of the elements of $\Gc$.

According to Lemma~\ref{choice} there exists an ordering of the variables such that with respect to the lexicographic order induced by this ordering the potential initial monomials become the initial monomials. Since initial monomials are relatively prime, it follows that $\Gc$ is a Gr\"obner basis of $I$, and since $x_{ij}$ does not divide any initial monomial of an element in $\Gc$ it follows that $x_{ij}$ that $x_{ij}$ is a nonzero divisor of $S/\ini(I)$. But then $x_{ij}$ is a nonzero divisor of $S/I$ as well.
\end{proof}

\section{Ideals generated by adjacent $2$-minors with quadratic Gr\"obner basis}

A configuration $\Pc$ of adjacent $2$-minors is called a {\em path}, if there exists ordering $\delta_1,\ldots,\delta_r$ of the elements of $\Pc$ such that for all \[
\delta_j\sect \delta_i\subset \delta_{i-1}\sect \delta_i\quad \text{for all} \quad j<i, \quad \text{and}\quad  \delta_{i-1}\sect \delta_i \quad\text{is an edge of $\delta_i$} .
\]
A path $\Pc$ with path ordering $\delta_1,\ldots,\delta_r$  where  $\delta_i=[a_i,a_i+1|b_i,b_i+1]$ for $i=1,\ldots,r$ is called {\em monotone}, if  the sequences of integers $a_1,\cdots, a_r$ and $b_1,\ldots, b_r$ are monotone sequences. The monotone path $\Pc$ is called {\em monotone increasing (decreasing)}  if the  sequence $b_1,\ldots, b_r$ is increasing (decreasing). We define the {\em end points} of $\Pc$ to be $(a_1,b_1)$ and $(a_r+1, b_r+1)$ if $\Pc$ is monotone increasing, and to be $(a_1,b_1+1)$ and $(a_r+1,b_r)$ if $\Pc$ is monotone decreasing. If for $\Pc$ we have $a_1=a_2=\cdots=a_r$, or $b_1=b_2=\cdots =b_r$, then we call $\Pc$ a {\em line path}. Notice that a line graph is both monotone increasing and monotone decreasing.

Let  $\delta =ad-bc$  be  an adjacent 2-minor as shown in Figure~\ref{antidiagonal}. Then we call the monomial $ad$ the {\em diagonal} of $\delta$ and the monomial  $bc$ the {\em anti-diagonal} of $\delta$.

\begin{figure}[hbt]
\begin{center}
\psset{unit=0.7cm}
\begin{pspicture}(4.5,0)(4.5,1.5)
\pspolygon[style=fyp,fillcolor=light](4,0)(4,1)(5,1)(5,0)
\rput(3.7,1){a}
\rput(3.7,0){c}
\rput(5.3,1){b}
\rput(5.3,0){d}
\end{pspicture}
\end{center}
\caption{}\label{antidiagonal}
\end{figure}

\begin{Lemma}
\label{diagonal}
Let $\Pc$ be a monotone increasing (decreasing) path of 2-minors.   Then for any monomial order $<$ for which $I(\Pc)$ has a quadratic Gr\"obner basis,  the initial monomials of the generators are all diagonals (anti-diagonals).
\end{Lemma}

\begin{proof}
Suppose first that $\Pc$ is a line path. If $I(\Pc)$ has a quadratic Gr\"obner basis,  then initial monomials of the $2$-minors of $\Pc$ are all diagonals or all anti-diagonals, because otherwise  there would be two 2-minors $\delta_1$ and $\delta_2$ in $\Pc$ connected by an edge such that $\ini(\delta_1)$ is a diagonal and $\ini(\delta_2)$ is an anti-diagonal. The $S$-polynomial of $\delta_1$ and $\delta_2$  is a binomial of degree 3 which belongs to the reduced Gr\"obner basis of $I$, a contradiction. If all initial monomials of the 2-minors in $\Pc$ are diagonals, we interpret $\Pc$ as a monotone increasing path, and  if all initial monomials of the 2-minors in $\Pc$ are anti-diagonals, we interpret $\Pc$ as a monotone decreasing path.

Now assume that $\Pc$ is not a line path. We may assume that $\Pc$ is monotone increasing. (The argument for a monotone decreasing path is similar). Then, since $\Pc$ is not a line path it contains one of the following sub-paths displayed in Figure~\ref{subpath}.

\begin{figure}[hbt]
\begin{center}
\psset{unit=0.7cm}
\begin{pspicture}(4.5,-1)(4.5,2)
\pspolygon[style=fyp,fillcolor=light](1,0)(1,1)(2,1)(2,0)
\pspolygon[style=fyp,fillcolor=light](2,0)(2,1)(3,1)(3,0)
\pspolygon[style=fyp,fillcolor=light](1,1)(1,2)(2,2)(2,1)
\pspolygon[style=fyp,fillcolor=light](6,0)(6,1)(7,1)(7,0)
\pspolygon[style=fyp,fillcolor=light](6,1)(6,2)(7,2)(7,1)
\pspolygon[style=fyp,fillcolor=light](7,1)(7,2)(8,2)(8,1)
\end{pspicture}
\end{center}
\caption{}\label{subpath}
\end{figure}

For both sub-paths the initial monomials must be diagonals, otherwise $I(\Pc)$ would not have a quadratic Gr\"obner basis. Then as in the case of line paths one sees that all the other initial monomials of $\Pc$ must be diagonals.
\end{proof}

A configuration of adjacent $2$-minors which are of the form shown in Figure~\ref{pinsaddle},   or are obtained by rotation from them,  are called {\em square}, {\em pin} and {\em saddle}, respectively.

\begin{figure}[hbt]
\begin{center}
\psset{unit=0.7cm}
\begin{pspicture}(5.5,-1)(5.5,4)
\pspolygon[style=fyp,fillcolor=light](-3,0)(-3,1)(-2,1)(-2,0)
\pspolygon[style=fyp,fillcolor=light](-2,0)(-2,1)(-1,1)(-1,0)
\pspolygon[style=fyp,fillcolor=light](-3,1)(-3,2)(-2,2)(-2,1)
\pspolygon[style=fyp,fillcolor=light](-2,1)(-2,2)(-1,2)(-1,1)
\rput(-2, -0.8){Square}
\pspolygon[style=fyp,fillcolor=light](2.8,0)(2.8,1)(3.8,1)(3.8,0)
\pspolygon[style=fyp,fillcolor=light](3.8,0)(3.8,1)(4.8,1)(4.8,0)
\pspolygon[style=fyp,fillcolor=light](4.8,0)(4.8,1)(5.8,1)(5.8,0)
\pspolygon[style=fyp,fillcolor=light](3.8,1)(3.8,2)(4.8,2)(4.8,1)
\rput(4.2, -0.8){Pin}
\pspolygon[style=fyp,fillcolor=light](9,0)(9,1)(10,1)(10,0)
\pspolygon[style=fyp,fillcolor=light](10,0)(10,1)(11,1)(11,0)
\rput(11.6,0.9){$\cdots$}\rput(12.4,0.9){$\cdots$}
\rput(11.6,0.1){$\cdots$}\rput(12.4,0.1){$\cdots$}
\pspolygon[style=fyp,fillcolor=light](13,0)(13,1)(14,1)(14,0)
\pspolygon[style=fyp,fillcolor=light](14,0)(14,1)(15,1)(15,0)
\pspolygon[style=fyp,fillcolor=light](9,1)(9,2)(10,2)(10,1)
\pspolygon[style=fyp,fillcolor=light](14,1)(14,2)(15,2)(15,1)
\rput(12, -0.8){Saddle}
\end{pspicture}
\end{center}
\caption{}\label{pinsaddle}
\end{figure}

\begin{Lemma}
\label{obvious}
Let $\Ac$ be a connected configuration of adjacent $2$-minors. Then $\Ac$ is a monotone path if and only if $\Ac$  contains  neither a square nor a pin nor a saddle.
\end{Lemma}

\begin{proof} Assume that $\Ac=\delta_1,\delta_2,\ldots,\delta_r$  with $\delta_i= [a_i,a_i+1|b_i,b_i+1]$ for $i=1,\ldots,r$ is a monotone path. Without loss of generality we may assume the both sequences $a_1,\cdots, a_r$ and $b_1,\ldots, b_r$  are monotone increasing. We will show by induction on $r$ that it contains no square, no pin and no saddle. For $r=1$ the statement is obvious. Now let us assume that the assertion is true for $r-1$. Since $\Ac'=\delta_1,\delta_2,\cdots,\delta_{r-1}$ is monotone increasing it follows that the coordinates of the minors $\delta_i$ for $i=1,\ldots,r-1$ sit inside the rectangle $R$ with corners  $(a_1,b_1), (a_{r-1}+1,b_1), (a_{r-1}+1,b_{r-1}+1), (a_1,b_{r-1}+1)$, and $\Ac'$ has no square, no pin and no saddle. Since $\Ac$ is monotone increasing, $\delta_r=[a_{r-1},a_{r-1}+1| b_{r-1}+1,b_{r-1}+2]$ or $\delta_r=[a_{r-1}+1,a_{r-1}+2| b_{r-1},b_{r-1}+1]$. It follows that if  $\Ac$ would contain
a square, a pin or a saddle, then the coordinates of one the minors $\delta_i$, $i=1,\ldots,r-1$ would not  be inside the rectangle $R$.

Conversely suppose that $\Ac$  contains   no square, no pin and no saddle. Then $\Ac'$ contains  no square, no pin and no saddle as well. Thus arguing by induction  on $r$, we may assume that $\Ac'$ is a monotone path. Without loss of generality we may even assume that $a_1\leq a_2\leq \cdots\leq a_{r-1}$ and $b_1\leq b_2\leq\cdots \leq b_{r-1}$. Now  let $\delta_r$  be connected to $\delta_i$ (via an edge). If $i\in \{2,\ldots,r-2\}$, then $\Ac$ contains a  square, a pin or a saddle which involves $\delta_r$, a contradiction. If $i=1$ or $i=r-1$,  and  $\Ac$ is not monotone, then contains $\Ac$ contains a square or a saddle involving $\delta_r$.
\end{proof}

With the notation introduced we have

\begin{Theorem}
\label{quadratic}
Let  $\Cc$ be a configuration of adjacent $2$-minors. Then the following conditions are equivalent:
\begin{enumerate}
\item[{\em (a)}] $I(\Cc)$ has a quadratic Gr\"obner basis with respect to the lexicographic order induced by a suitable order of the variables.
\item[{\em (b)}]
\begin{enumerate}
\item[{\em (i)}] Each connected component of $\Cc$ is a monotone path.
\item[{\em (ii)}] If $\Ac$ and $\Bc$ are components of $\Cc$ which meet in a vertex which is not  an end point  of $\Ac$ or not and end point  of $\Bc$, and if $\Ac$ is monotone increasing, then $\Bc$ must be monotone decreasing,  and vice versa.
\end{enumerate}
\item[{\em (c)}] The initial ideal of $I(\Cc)$ with respect to  the lexicographic order induced by a suitable order of the variables is a complete intersection.
\end{enumerate}
\end{Theorem}

\begin{proof}
(a) $\Rightarrow$ (b): (i) Suppose there is component $\Ac$ of $\Cc$ which is not a monotone path. Then, according to Lemma~\ref{obvious}, $\Ac$ contains a square, a pin or a saddle. In all three cases, no matter how we label  the vertices of the component $\Ac$, it will contain, up to a rotation or reflection,  two adjacent $2$-minors with leading terms as indicated in  Figure~\ref{leadingterm}.

\begin{figure}[hbt]
\begin{center}
\psset{unit=0.7cm}
\begin{pspicture}(4.5,-1)(4.5,2)
\pspolygon[style=fyp,fillcolor=light](1,0)(1,1)(2,1)(2,0)
\pspolygon[style=fyp,fillcolor=light](2,0)(2,1)(3,1)(3,0)
\psline(1,1)(2,0)
\psline(2,0)(3,1)
\pspolygon[style=fyp,fillcolor=light](6,0)(6,1)(7,1)(7,0)
\pspolygon[style=fyp,fillcolor=light](7,1)(7,2)(8,2)(8,1)
\psline(6,0)(7,1)
\psline(7,1)(8,2)
\rput(1,1.3){a}
\rput(2,1.3){b}
\rput(3,1.3){c}
\rput(1,-0.3){d}
\rput(2,-0.3){e}
\rput(3,-0.3){f}
\rput(7,2.3){a}
\rput(8,2.3){b}
\rput(5.7,1){c}
\rput(6.7,1.3){d}
\rput(8.2,1){e}
\rput(6,-0.3){f}
\rput(7,-0.3){g}
\end{pspicture}
\end{center}
\caption{}\label{leadingterm}
\end{figure}

In the first case the $S$-polynomial of the two minors is   $abf-bcd$ and in the second case it is $aef-bcg$. We claim that in both cases these binomials belong to the reduced Gr\"obner basis of $I(\Cc)$, which contradicts our assumption (a).

Indeed, first observe that the adjacent 2-minors generating the ideal $I(\Cc)$ is the unique minimal set of binomials generating $I(\Cc)$. Therefore, the initial monomials of degree 2 are exactly the initial monomials of these binomials. Suppose now that  $abf-bcd$ does not belong to the reduced Gr\"obner basis of $I$, then one of the monomials $ab$, $af$ or $bf$ must be the leading of an adjacent 2-minor, which is impossible. In the same way one argues in the second case.

(ii) Assume $\Ac$ and   $\Bc$ have  a vertex $c$ in common. Then $c$ must be a corner of $\Ac$ and $\Bc$, that is, a vertex which belongs to exactly one $2$-minor of $\Ac$ and exactly one $2$-minor of $\Bc$, see Figure~\ref{commoncorner}.

\begin{figure}[hbt]
\begin{center}
\psset{unit=0.7cm}
\begin{pspicture}(4.5,0)(4.5,4)
\pspolygon[style=fyp,fillcolor=light](3.5,2)(4.5,2)(4.5,3)(3.5,3)
\pspolygon[style=fyp,fillcolor=light](3.5,3)(4.5,3)(4.5,4)(3.5,4)
\pspolygon[style=fyp,fillcolor=light](2.5,2)(3.5,2)(3.5,3)(2.5,3)
\pspolygon[style=fyp,fillcolor=light](5.5,0)(5.5,1)(6.5,1)(6.5,0)
\pspolygon[style=fyp,fillcolor=light](4.5,1)(4.5,2)(5.5,2)(5.5,1)
\pspolygon[style=fyp,fillcolor=light](5.5,1)(5.5,2)(6.5,2)(6.5,1)
\rput(4.8,2.3){c}
\end{pspicture}
\end{center}
\caption{}\label{commoncorner}
\end{figure}

If for both components of  the initial monomials are the diagonals (anti-diagonals), then the $S$-polynomial of the 2-minor in $\Ac$ with vertex $c$ and the $2$-minor of $\Bc$ with vertex $c$ is a binomial of degree three whose initial monomial is not divisible by any initial monomial of the generators of $\Cc$, unless $c$ is an end point of both $\Ac$ and $\Bc$. Thus the desired conclusion follows from  Lemma~\ref{diagonal}.

(b) $\Rightarrow$  (c): The condition (b) implies that any pair of  initial monomials of two distinct binomial generators of $I(\Cc)$ are relatively prime. Hence the initial ideal is a complete intersection.

(c)  $\Rightarrow$ (d): Since the initial monomial of the 2-minors generating $I(\Cc)$ belong to any reduced Gr\"obner basis of $I(\Cc)$, they must form a regular sequence. This implies that $S$-polynomials of any two generating 2-minors of $I(\Cc)$ reduce to $0$. Therefore $I(\Cc)$ has a quadratic Gr\"obner basis.
\end{proof}

\begin{Corollary}
\label{radicalideal}
Let $\Cc$ be a configuration  satisfying the  conditions of Theorem~\ref{quadratic}(b). Then $I(\Cc)$  is a radical ideal generated by a regular sequence.
\end{Corollary}

\begin{proof} Let $\Cc=\delta_1,\ldots,\delta_r$. By Theorem~\ref{quadratic} there exist a monomial order $<$ such that $\ini_<(\delta_1),\ldots, \ini_<(\delta_r)$ is a regular sequence. It follows that $\delta_1,\ldots,\delta_r$  is a regular sequence. Since the initial monomials are squarefree and form a Gr\"obner basis of $I(\Cc)$ it follows that $I(\Cc)$ is a radical ideal, see for example \cite[Proof of Corollary 2.2]{HH}.
\end{proof}

To demonstrate Theorem~\ref{quadratic} we consider the following two examples displayed in Figure~\ref{two}

\begin{figure}[hbt]
\begin{center}
\psset{unit=0.7cm}
\begin{pspicture}(4.5,0)(4.5,4)
\pspolygon[style=fyp,fillcolor=light](7,2)(8,2)(8,3)(7,3)
\pspolygon[style=fyp,fillcolor=light](7,3)(8,3)(8,4)(7,4)
\pspolygon[style=fyp,fillcolor=light](6,2)(7,2)(7,3)(6,3)
\pspolygon[style=fyp,fillcolor=light](8,0)(8,1)(9,1)(9,0)
\pspolygon[style=fyp,fillcolor=light](8,1)(8,2)(9,2)(9,1)
\pspolygon[style=fyp,fillcolor=light](9,1)(9,2)(10,2)(10,1)
\rput(5.5,3){$\Ac$}
\rput(10.5,1){$\Bc$}

\pspolygon[style=fyp,fillcolor=light](-1,1)(0,1)(0,2)(-1,2)
\pspolygon[style=fyp,fillcolor=light](-1,0)(0,0)(0,1)(-1,1)
\pspolygon[style=fyp,fillcolor=light](-1,3)(0,3)(0,4)(-1,4)
\pspolygon[style=fyp,fillcolor=light](0,1)(1,1)(1,2)(0,2)

\pspolygon[style=fyp,fillcolor=light](-2,2)(-1,2)(-1,3)(-2,3)

\pspolygon[style=fyp,fillcolor=light](-2,3)(-1,3)(-1,4)(-2,4)
\rput(-2.5,3){$\Ac$}
\rput(1.5,1){$\Bc$}
\end{pspicture}
\end{center}
\caption{}\label{two}
\end{figure}

In both   examples the component $\Ac$ and the component $\Bc$ are monotone increasing paths. In the first example  $\Ac$ and $\Bc$ meet in a vertex which is an end point of $\Ac$, therefore condition (b)(ii) of Theorem~\ref{quadratic} is satisfied, and the ideal $I(\Ac\union \Bc)$ has a quadratic Gr\"obner basis. However in the second example  $\Ac$ and $\Bc$  meet in a vertex which is not and end point of $\Ac$ and not and end point of $\Bc$. Therefore condition (b)(ii) of Theorem~\ref{quadratic} is not satisfied, and the ideal $I(\Ac\union \Bc)$ does not have  a quadratic Gr\"obner basis for the lexicographic  order induced by any order of the variables.

\section{Minimal prime ideals of special configurations of adjacent $2$-minors}

Let $\Cc$ be a connected configuration of adjacent 2-minors.  In this section we want to describe a primary decomposition of $\sqrt{I(\Cc)}$. For this purpose we have to introduce some terminology: let $\Cc=\delta_1,\delta_2,\ldots,\delta_r$ be an arbitrary configuration of adjacent $2$-minors. A subset $W$ of the vertex set of $\Cc$ is called {\em admissible}, if for each index $i$  either $W\sect V(\delta_i)=\emptyset$ or $W\sect V(\delta_i)$ contains an edge of $\delta_i$. For example, the admissible sets of the configuration shown in Figure~\ref{admissible}

\begin{figure}[hbt]
\begin{center}
\psset{unit=0.7cm}
\begin{pspicture}(4.5,-1)(4.5,4)
\pspolygon[style=fyp,fillcolor=light](2.8,0)(2.8,1)(3.8,1)(3.8,0)
\pspolygon[style=fyp,fillcolor=light](3.8,0)(3.8,1)(4.8,1)(4.8,0)
\pspolygon[style=fyp,fillcolor=light](4.8,1)(4.8,2)(5.8,2)(5.8,1)
\pspolygon[style=fyp,fillcolor=light](3.8,1)(3.8,2)(4.8,2)(4.8,1)
\rput(3.8,2.3){a}
\rput(4.8,2.3){b}
\rput(5.8,2.3){c}
\rput(2.6,1.3){d}
\rput(3.5,1.3){e}
\rput(5,0.7){f}
\rput(6.1,1){g}
\rput(2.8,-0.3){h}
\rput(3.8,-0.3){i}
\rput(4.8,-0.3){j}
\end{pspicture}
\end{center}
\caption{}\label{admissible}
\end{figure}
are the following
\[
\emptyset, \{c,g\}, \{d,h\}, \{a,e,i\},\{b,f,j\},\{a,b,c\},\ldots,\{a,b,c,d,e,f,g,h,i,j\}.
\]
Let $W\subset V(\Cc)$ be an admissible set. We define an ideal $P_W(\Cc)$ containing $I(\Cc)$ as follows: the generators of $P_W(\Cc)$ are the variables belonging to $W$ and all $2$-minors $\delta=[a_1,a_2|b_1,b_2]$ (not necessarily adjacent) such that all vertices $(i,j)$ with $a_1\leq i\leq a_2$  and $b_1\leq j\leq b_2$ belong to $V(\Cc)\setminus W$.

Note that  $P_W(\Cc)=(W, P_\emptyset(\Cc'))$ where  $\Cc'=\{\delta\in \Cc\:\, V(\delta)\sect W=\emptyset\}$. We denote by $\Gc(\Cc')$ the set of $2$-minors of $P_\emptyset(\Cc')$ and call it the set of {\em inner $2$-minors}  of $\Cc'$.

For example, if we take the configuration displayed in Figure~\ref{admissible}, then
\begin{eqnarray*}
P_{\emptyset}(\Cc)&=&(af-be, aj-bi, ej-fi, ag-ce, bg-cf, di-eh, dj-fh), \\
P_{\{d,h\}}(\Cc)&=&(d,h,af-be, aj-bi, ej-fi, ag-ce, bg-cf).
\end{eqnarray*}

\begin{Lemma}
\label{primeadmissible}
Let $\Cc$ be a special  configuration of adjacent $2$-minors. Then for any admissible set $W\subset V(\Pc)$ the ideal $P_W(\Cc)$ is a prime ideal.
\end{Lemma}

\begin{proof}
Let   $P_W(\Cc)=P_W(\Cc)=(W, P_\emptyset(\Cc'))$,
where $\Cc'=\{\delta\in \Cc\:\, V(\delta)\sect W=\emptyset\}$.

Note that $\Cc'$  is again a special configuration of $2$-adjacent minors. Indeed,
let $\delta_1,\delta_2\in\Cc'$ be two adjacent $2$-minors with exactly one common vertex. Since $\Cc$ is special, there exists $\delta\in \Cc$ which has a common edge  with $\delta_1$ and a common edge with $\delta_2$. Since $\delta\not\in \Cc'$, the set $W$ contains an edge of $\delta$. This implies that $W\sect V(\delta_1)\neq \emptyset$ or $W\sect V(\delta_2)\neq\emptyset$, contradicting the fact that $V(\Cc')\sect W=\emptyset$.

By a result of Qureshi \cite{Q},   $P_\emptyset(\Cc')$ is a prime ideal. Therefore $P_W(\Cc)$ is a prime ideal.
\end{proof}

\begin{Theorem}
\label{primeintersection}
Let $\Cc$ be a special configuration of adjacent $2$-minors.  Then
\[
\sqrt{I(\Cc)}=\Sect_{W}P_W(\Cc),
\]
where the intersection is taken over all admissible sets $W\subset V(\Cc)$.
\end{Theorem}

\begin{proof}
We show that if $P$ is a minimal  prime ideal of $I(\Cc)$, then there exists an admissible set $W\subset V(\Cc)$ such that $P=P_W(\Cc)$.

So now let $P$ be any minimal prime ideal of $I(\Cc)$, and let $W$ be the set of variables among the generators of $P$. We claim that $W$ is admissible. Indeed,  suppose that $W\sect V(\delta)\neq \emptyset$ for  some adjacent $2$-minor of $\Cc$. Say, $\delta=ad-bc$ and $a\in W$. Then $bc\in P$. Hence, since $P$ is a prime ideal, it follows that $b\in P$ or $c\in P$. Thus $W$ contains the edge $\{a,c\}$ or the edge $\{a,b\}$ of $\delta$.

Since $I(\Cc)\subset P$ it follows that $(W,I(\Cc))\subset P$. Observe that $(W,I(\Cc))=(W, I(\Cc'))$,  where $W\sect V(\Cc')=\emptyset$ and $\Cc'$ is again a special configuration, see the proof of Lemma~\ref{primeadmissible}.   Modulo $W$ we obtain a minimal prime ideal $\bar{P}$,  which contains no variables, of the ideal $I(\Cc')$.

By the result of Qureshi \cite{Q} the ideal $P_\emptyset(\Cc')$ is a prime ideal containing  $I(\Cc')$. Thus the assertion of the theorem follows once we have shown that $P_\emptyset(\Cc')\subset \bar{P}$.

Since $P_\emptyset(\Cc')$ is generated by the union of the set of  $2$-minors of certain $r\times s$-matrices, it suffices to show that if $P$ is a prime ideal having no variables among its generators and containing all adjacent $2$-minors of the $r\times s$-matrix $X$, then it contains all $2$-minors of $X$. In order to prove this,  let $\delta=[a_1,a_2|b_1,b_2]$ be an arbitrary $2$-minor of $X$. We prove that $\delta\in P$ by induction on $(a_2-a_1)+(b_2-b_1)$. For $(a_2-a_1)+(b_2-b_1)=2$, this is the case by assumption. Now let  $(a_2-a_1)+(b_2-b_1)>2$. We may assume that $a_2-a_1>1$.  Let $\delta_1=[a_1,a_2-1|b_1,b_2]$ and $\delta_2=[a_2-1,a_2|b_1,b_2]$. Then $x_{a_2-1,b_1}\delta= x_{a_2,b_1}\delta_1+x_{a_1,b_1}\delta_2$. Therefore, by induction hypothesis $x_{a_2-1,b_1}\delta \in P$. Since $P$ is a prime ideal, and $x_{a+k-1,1}\not\in P$ it follows that $\delta\in P$, as desired.
\end{proof}

In general it seems to be pretty hard to find the primary decomposition for ideals generated by adjacent $2$-minors. This seems to be even difficult for ideals described in Theorem~\ref{quadratic}. For example, the primary decomposition (computed with the help of Singular) of the ideal $I(\Cc)$ of adjacent $2$-minors shown in Figure~\ref{difficult} is the following:
\begin{figure}[hbt]
\begin{center}
\psset{unit=0.7cm}
\begin{pspicture}(4.5,-1)(4.5,3)
\pspolygon[style=fyp,fillcolor=light](2.8,0)(2.8,1)(3.8,1)(3.8,0)
\pspolygon[style=fyp,fillcolor=light](3.8,-1)(3.8,0)(4.8,0)(4.8,-1)
\pspolygon[style=fyp,fillcolor=light](4.8,0)(4.8,1)(5.8,1)(5.8,0)
\pspolygon[style=fyp,fillcolor=light](3.8,1)(3.8,2)(4.8,2)(4.8,1)
\rput(3.8,2.3){a}
\rput(4.8,2.3){b}
\rput(2.6,1.3){c}
\rput(3.5,1.4){d}
\rput(5.1,1.3){e}
\rput(6.1,1.2){f}
\rput(2.6,-0.3){g}
\rput(3.5,-0.3){h}
\rput(5.1,-0.3){i}
\rput(6.1,-0.2){j}
\rput(3.8,-1.3){k}
\rput(4.8,-1.3){l}
\end{pspicture}
\end{center}
\caption{}\label{difficult}
\end{figure}
\begin{eqnarray*}
I(\Cc)&=&(ae-bd,ch-dg,ej-fi, hl-ik)\\
&= &(ik-hl, fi-ej,  dg-ch,  bd-ae,  bcjk-afgl)\sect (d,e,h,i).
\end{eqnarray*}

It turns out that $I(\Cc)$ is a radical ideal.
On the other hand, if we add the minor $di-eh$ we get a connected configuration $\Cc'$ of adjacent $2$-minors. The ideal $I(\Cc')$ is not radical, because it contains a pin, see Proposition~\ref{onedirection}. Indeed, one has
\begin{eqnarray*}
\sqrt{I(\Cc')}&= &(ae-bd,ch-dg,ej-fi,hl-ik, di-eh,fghl-chjl,\\
&& bfhl-aejl,bchk-achl, bcfh-acej)\\
\end{eqnarray*}
Applying  Proposition~\ref{minimalprimes}, we get
\begin{eqnarray*}
\sqrt{I(\Cc')}&= &(ae-bd,ch-dg,ej-fi,hl-ik, di-eh,fghl-chjl,\\
&& bfhl-aejl,bchk-achl, bcfh-acej)\\
&=&(-ik+hl, -fi+ej, -ek+dl, -fh+dj, -eh+di, -fg+cj, -eg+ci,\\
&&-dg+ch, -bk+al, -bh+ai, -bd+ae)\\
&\sect&(d,e,h,i)\sect (a,d,h,i,j)\sect (d,e,f,h,k)\sect(c,d,e,i,l)\sect(b,e,g,h,i) \\
&\sect&(a,d,h,k,ej-fi)\sect (c,d,e,f,hl-ik)\sect(b,e,i,l,ch-dg)\\
&\sect& (g,h,i,j, ae-bd).
\end{eqnarray*}

The presentation of $\sqrt{I(\Cc)}$ as an intersection of prime ideals as given in Theorem~\ref{primeintersection} is usually not irredundant. In order to obtain an irredundant intersection, we have to identify the minimal prime ideals of $I(\Cc)$ among the prime ideals $P_W(\Cc)$.

\begin{Theorem}
\label{minimalprimes}
Let $\Cc$ be a special configuration of adjacent $2$-minors,  and let $V,W\subset V(\Cc)$ be admissible sets of $\Cc$, and let $P_V(\Cc)=(V,\Gc(\Cc'))$ and $P_W(\Cc)=(W,\Gc(\Cc''))$, as given in Lemma~\ref{primeadmissible}. Then
\begin{enumerate}
\item[{\em (a)}] $P_V(\Cc)\subset P_W(\Cc)$ if and only if $V\subset W$, and for all elements $$\delta\in \Gc(\Cc')\setminus \Gc(\Cc'')$$  one has that $W\sect V(\delta)$ contains an edge of $\delta$.
\item[{\em (b)}] $P_W(\Cc)=(W,\Gc(\Cc''))$ is a minimal prime ideal of $I(\Cc)$ if and only if for all admissible subsets $V\subset W$ with $P_V(\Cc)=(V,\Gc(\Cc'))$ there exists  $$\delta \in \Gc(\Cc')\setminus \Gc(\Cc'')$$ such that the set $W\sect V(\delta)$ does not contain and edge of $\delta$.
\end{enumerate}
\end{Theorem}

\begin{proof}
(a) Suppose that  $P_V(\Cc)\subset P_W(\Cc)$. The only variables in $P_W(\Cc)$ are those belonging to $W$. This shows that  $V\subset W$. The inclusion  $P_V(\Cc)\subset P_W(\Cc)$ implies that $\delta\in (W, \Gc(\Cc''))$ for all $\delta \in \Gc(\Cc')$. Suppose
$W\sect V(\delta)=\emptyset$. Then $\delta$ belongs to  $P_{\emptyset}(\Cc'')=(\Gc(\Cc''))$. Let $f=u-v\in \Gc(\Cc'')$. Neither $u$ now $v$ appears in another element of $\Gc(\Cc'')$. Therefore any binomial of degree $2$ in   $P_{\emptyset}(\Cc'')$ belongs to  $\Gc(\Cc'')$. In particular,   $\delta\in  \Gc(\Cc'')$, a contradiction. Therefore, $W\sect V(\delta)\neq \emptyset$. Suppose that $W\sect V(\delta)$ does not contain an edge of $\delta=ad-bc$. We may assume that $a\in W\sect V(\delta)$.  Then, since $\delta\in P_W(\Cc)$, it follows that $bc\in P_W(\Cc)$. Since $P_W(\Cc)$ is a prime ideal, we  conclude  that $b\in P_W(\Cc)$ or $c\in P_W(\Cc)$.Then $b\in W$ or $c\in W$ and hence either the edge $\{a,b\}$ or the edge $\{a,c\}$ belongs to $W\sect V(\delta)$.

The `if' part of statement (a) is obvious.

(b) is a simple consequence of Theorem~\ref{primeintersection} and statement (a).
\end{proof}

In Figure~\ref{minimalp} we display all the minimal prime ideals $I(\Pc)$ for the  path $\Pc$ shown in Figure~\ref{admissible}. The fat dots mark the admissible sets and the dark shadowed areas, the regions where the inner $2$-minors  have to be taken.

\begin{figure}[hbt]
\begin{center}
\psset{unit=0.7cm}
\begin{pspicture}(4.5,-3)(4.5,4)
\pspolygon[style=fyp,fillcolor=dark](-5,0)(-5,1)(-4,1)(-4,0)
\pspolygon[style=fyp,fillcolor=dark](-4,0)(-4,1)(-3,1)(-3,0)
\pspolygon[style=fyp,fillcolor=dark](-3,1)(-3,2)(-2,2)(-2,1)
\pspolygon[style=fyp,fillcolor=dark](-4,1)(-4,2)(-3,2)(-3,1)

\pspolygon[style=fyp,fillcolor=dark](0,0)(0,1)(1,1)(1,0)
\pspolygon[style=fyp,fillcolor=light](1,0)(1,1)(2,1)(2,0)
\pspolygon[style=fyp,fillcolor=light](2,1)(2,2)(3,2)(3,1)
\pspolygon[style=fyp,fillcolor=light](1,1)(1,2)(2,2)(2,1)
\rput(2,0){$\bullet$}
\rput(2,1){$\bullet$}
\rput(2,2){$\bullet$}

\pspolygon[style=fyp,fillcolor=light](4,0)(4,1)(5,1)(5,0)
\pspolygon[style=fyp,fillcolor=light](5,0)(5,1)(6,1)(6,0)
\pspolygon[style=fyp,fillcolor=dark](6,1)(6,2)(7,2)(7,1)
\pspolygon[style=fyp,fillcolor=light](5,1)(5,2)(6,2)(6,1)
\rput(5,0){$\bullet$}
\rput(5,1){$\bullet$}
\rput(5,2){$\bullet$}

\pspolygon[style=fyp,fillcolor=light](9,0)(9,1)(10,1)(10,0)
\pspolygon[style=fyp,fillcolor=light](10,0)(10,1)(11,1)(11,0)
\pspolygon[style=fyp,fillcolor=light](11,1)(11,2)(12,2)(12,1)
\pspolygon[style=fyp,fillcolor=light](10,1)(10,2)(11,2)(11,1)
\rput(10,0){$\bullet$}
\rput(10,1){$\bullet$}
\rput(11,1){$\bullet$}
\rput(12,1){$\bullet$}

\pspolygon[style=fyp,fillcolor=light](-5,-3)(-5,-2)(-4,-2)(-4,-3)
\pspolygon[style=fyp,fillcolor=light](-4,-3)(-4,-2)(-3,-2)(-3,-3)
\pspolygon[style=fyp,fillcolor=light](-3,-2)(-3,-1)(-2,-1)(-2,-2)
\pspolygon[style=fyp,fillcolor=light](-4,-2)(-4,-1)(-3,-1)(-3,-2)
\rput(-4,-3){$\bullet$}
\rput(-4,-2){$\bullet$}
\rput(-3,-2){$\bullet$}
\rput(-3,-1){$\bullet$}

\pspolygon[style=fyp,fillcolor=light](0,-3)(0,-2)(1,-2)(1,-3)
\pspolygon[style=fyp,fillcolor=light](1,-3)(1,-2)(2,-2)(2,-3)
\pspolygon[style=fyp,fillcolor=light](2,-2)(2,-1)(3,-1)(3,-2)
\pspolygon[style=fyp,fillcolor=light](1,-2)(1,-1)(2,-1)(2,-2)
\rput(0,-2){$\bullet$}
\rput(1,-2){$\bullet$}
\rput(2,-2){$\bullet$}
\rput(3,-2){$\bullet$}

\pspolygon[style=fyp,fillcolor=light](4,-3)(4,-2)(5,-2)(5,-3)
\pspolygon[style=fyp,fillcolor=light](5,-3)(5,-2)(6,-2)(6,-3)
\pspolygon[style=fyp,fillcolor=light](6,-2)(6,-1)(7,-1)(7,-2)
\pspolygon[style=fyp,fillcolor=light](5,-2)(5,-1)(6,-1)(6,-2)
\rput(4,-2){$\bullet$}
\rput(5,-2){$\bullet$}
\rput(6,-2){$\bullet$}
\rput(6,-1){$\bullet$}

\end{pspicture}
\end{center}
\caption{}\label{minimalp}
\end{figure}

\section{Special configurations which are radical}

This section is devoted to study  special  configuration of adjacent $2$-minors $\Cc$ for which $I(\Cc)$ is a radical ideal. Any special configuration $\Cc$ is a disjoint union $\Union_{i=1}^k\Cc^i$ of connected special configurations of adjacent $2$-minors. It follows that $I(\Cc)$ is radical if and only if each $I(\Cc^i)$ is radical. Thus when we discuss the radical property we may always assume that $\Cc$ is connected.

We call a configuration $\Cc$ of adjacent $2$-minors a {\em cycle}, if each for each $\delta\in \Cc$ there exist exactly two $\delta_1,\delta_2\in \Cc$ such that $\delta$ and $\delta_1$ have a common edge and $\delta$ and $\delta_2$ have a common edge.

\begin{Lemma}
\label{pinlemma}
Let $\Cc$ be a connected special configuration which does not contain a pin. Then $\Cc$ is a path or a cycle.
\end{Lemma}
\begin{proof} If $\Cc$ does not contain a pin, then for each adjacent $2$-minor $\delta \in \Cc$ there exists at most two adjacent $2$-minors in $\Cc$ which have a common edge which have a common edge with $\delta$. Thus, if $\Cc$ is not a cycle but connected, there exists $\delta_1,\delta_2\in \Cc$ such that $\delta_1$ has a common edge only with $\delta_1$. Now in the configuration $\Cc' =\Cc\setminus \{\delta_1\}$  the element $\delta_2$ has at most one edge in common with another element of $\Cc'$. If $\delta_2$ has no edge in common with another element of $\Cc'$, then $\Cc=\{\delta_1,\delta_2\}$. Otherwise, continuing this argument, a simple induction argument yields the desired conclusion.
\end{proof}

\begin{Proposition}
\label{onedirection}
Let $\Cc$  be a connected special configuration of adjacent $2$-minors. If  $I(\Cc)$ is a radical ideal, then   $\Cc$ is a path or a cycle.
\end{Proposition}

\begin{proof}
By Lemma~\ref{pinlemma} it is enough to prove that $\Cc$ does not contain a pin. Suppose $\Cc$ contains the  pin $\Cc'$ as in Figure~\ref{pininside}.
\begin{figure}[hbt]
\begin{center}
\psset{unit=0.7cm}
\begin{pspicture}(4.5,-1)(4.5,4)
\pspolygon[style=fyp,fillcolor=light](2.8,0)(2.8,1)(3.8,1)(3.8,0)
\pspolygon[style=fyp,fillcolor=light](3.8,0)(3.8,1)(4.8,1)(4.8,0)
\pspolygon[style=fyp,fillcolor=light](4.8,0)(4.8,1)(5.8,1)(5.8,0)
\pspolygon[style=fyp,fillcolor=light](3.8,1)(3.8,2)(4.8,2)(4.8,1)

\rput(3.8,2.3){$a$}
\rput(4.8,2.3){$b$}
\rput(2.65,1.25){$c$}
\rput(6,1.3){$f$}
\rput(2.8,-0.3){$g$}
\rput(3.8,-0.3){$h$}
\rput(4.8,-0.3){$i$}
\rput(5.8,-0.3){$j$}
\rput(3.5,1.35){$d$}
\rput(5.1,1.3){$e$}
\end{pspicture}
\end{center}
\caption{}\label{pininside}
\end{figure}
Then $q=acej-bcfh\not\in I(\Cc')$ but $q^2\in I(\Cc')\subset I(\Cc)$. We consider two cases. In the first case suppose that the adjacent $2$-minors $kd-ac$ and $bf-le$ do not belong to $\Cc$, see Figure~\ref{notbelong}.

\begin{figure}[hbt]
\begin{center}
\psset{unit=0.7cm}
\begin{pspicture}(4.5,-1)(4.5,4)
\pspolygon[style=fyp,fillcolor=light](2.8,0)(2.8,1)(3.8,1)(3.8,0)
\pspolygon[style=fyp,fillcolor=light](3.8,0)(3.8,1)(4.8,1)(4.8,0)
\pspolygon[style=fyp,fillcolor=light](4.8,0)(4.8,1)(5.8,1)(5.8,0)
\pspolygon[style=fyp,fillcolor=light](3.8,1)(3.8,2)(4.8,2)(4.8,1)
\pspolygon[style=fyp,fillcolor=light](2.8,1)(2.8,2)(3.8,2)(3.8,1)
\pspolygon[style=fyp,fillcolor=light](4.8,1)(4.8,2)(5.8,2)(5.8,1)
\rput(3.8,2.3){$a$}
\rput(4.8,2.3){$b$}
\rput(2.5,1){$c$}
\rput(6.1,1){$f$}
\rput(2.8,-0.3){$g$}
\rput(3.8,-0.3){$h$}
\rput(4.8,-0.3){$i$}
\rput(5.8,-0.3){$j$}
\rput(3.5,1.35){$d$}
\rput(5.1,1.3){$e$}
\rput(2.8,2.3){$k$}
\rput(5.8,2.3){$l$}
\end{pspicture}
\end{center}
\caption{}\label{notbelong}
\end{figure}

Then $q\not\in (I(\Cc),W)$ where $W$ is the set of vertices which do not belong to $\Cc'$. It follows that $q\not\in I(\Cc)$. In the second case we may assume that $ac-kd\in \Cc$.
Let $\Cc''$ be the configuration with the adjacent $2$-minors $kd-ac,ae-bd,ch-dg,di-eh$. Then $r=kdi-aeg\not\in I(\Cc'')$ but $r^2\in I(\Cc'')\subset I(\Cc)$. Then $r\not \in (I(\Cc),V)$ where $V$ is the set of vertices in $\Cc$ which do not belong to $\Cc'$. It follows that $r\not\in I(\Cc)$. Thus in both cases we see that $I(\Cc)$ is not a radical ideal.
\end{proof}

For the cycle $\Cc$ displayed in Figure~\ref{cycle} the ideal $I(\Cc)$ is not radical.
Indeed we have $f=b^2hino-abhjno\not\in I(\Cc)$, but $f^2\in I(\Cc)$. By computational evidence, we expect that the  ideal of adjacent $2$-minors of a cycle is a radical ideal if and only if the length of the cycle is $\geq 12$.
\begin{figure}[hbt]
\begin{center}
\psset{unit=0.7cm}
\begin{pspicture}(-0.5,1)(-0.5,5)
\pspolygon[style=fyp,fillcolor=light](-1,1)(0,1)(0,2)(-1,2)
\pspolygon[style=fyp,fillcolor=light](-2,1)(-1,1)(-1,2)(-2,2)
\pspolygon[style=fyp,fillcolor=light](0,1)(1,1)(1,2)(0,2)
\pspolygon[style=fyp,fillcolor=light](-2,2)(-1,2)(-1,3)(-2,3)
\pspolygon[style=fyp,fillcolor=light](-2,3)(-1,3)(-1,4)(-2,4)
\pspolygon[style=fyp,fillcolor=light](-1,3)(0,3)(0,4)(-1,4)
\pspolygon[style=fyp,fillcolor=light](0,3)(1,3)(1,4)(0,4)
\pspolygon[style=fyp,fillcolor=light](0,2)(1,2)(1,3)(0,3)
\rput(-2,4.3){$a$}
\rput(-1,4.3){$b$}
\rput(0,4.3){$c$}
\rput(1,4.3){$d$}
\rput(-2.3,3){$e$}
\rput(-1.3,3.3){$f$}
\rput(0.3,3.3){$g$}
\rput(1.3,3){$h$}
\rput(-2.3,2){$i$}
\rput(-1.3,1.7){$j$}
\rput(0.3,1.7){$k$}
\rput(1.3,2){$l$}
\rput(-2,0.7){$m$}
\rput(-1,0.7){$n$}
\rput(0,0.7){$o$}
\rput(1,0.7){$p$}
\end{pspicture}
\end{center}
\caption{}\label{cycle}
\end{figure}
On the other hand, if  $\Pc$ is a monotone path, we know from Theorem~\ref{quadratic} that $I(\Pc)$ has a squarefree initial ideal. This implies that $I(\Pc)$ is a radical ideal. More generally, we expect that ideal of adjacent $2$-minors of a path  $\Pc$  is always a radical ideal, and prove this under the assumption that the ideal $I(\Pc)$  has no embedded prime ideals. In \cite[Theorem 4.2]{DES} the primary decomposition of the ideal of adjacent $2$-minors of a $4\times 4$-matrix is given, from which it can be seen that in general the  ideal of adjacent $2$-minors of a special configuration may have embedded prime ideals.

\begin{Theorem}
\label{radicalpath}
Let $\Pc$ be path, and suppose that $I(\Pc)$ has no embedded prime ideals. Then $I(\Pc)$ is a radical ideal.
\end{Theorem}

The proof will require several steps.

\begin{Lemma}
\label{easy}
Let $I\subset S$ be an ideal, and let $a, b\in S$ such that $a$ is a nonzero divisor modulo $(b,I)$. Then
\[
(ab,I)=(a,I)\sect(b,I).
\]
\end{Lemma}

\begin{proof}
Obviously one has $(ab,I)\subset (a,I)\sect(b,I)$. Conversely, let $f\in(a,I)\sect(b,I)$. Then
\[
f=ag_1+h_1=bg_2+h_2 \quad \text{with}\quad g_1,g_2\in S\quad \text{and}\quad h_1,h_2\in I.
\]
Therefore, $ag_1\in (b,I)$. Since $a$ is a nonzero divisor modulo $(b,I)$, it follows that $g_1=cb+h$ for some $c\in S$ and $h\in I$. Hence we get that $f=ag_1+h_1=a(cb+h)+h_1=abc+(ah+h_1)$. Thus $f\in (ab,I)$
\end{proof}

\begin{Lemma}
\label{nonzerodivisor}
Let $\Pc= \delta_1,\delta_2,\ldots, \delta_r$ be a path which is not  a line, and let $i>1$ be the smallest index for which $\delta_i$ has a free vertex $c$ (which we call a corner of the path). Let $a$ be the free vertex of $\delta_1$ whose first or second coordinate  coincides with that of $c$, see Figure~\ref{smallest}. Then $a$ does not belong to any minimal prime ideal of $I(\Pc)$.

In particular, if $I(\Pc)$ has no embedded prime ideals, the element $a$ is a nonzero divisor of $S/I(\Pc)$.
\end{Lemma}

\begin{proof}
We may assume that, like in Figure~\ref{smallest}, the first $2$-adjacent minors up to the first corner form a horizontal path. Let $W$ be an admissible set with $a\in W$. In our discussion we refer to the notation given in Figure~\ref{smallest}.  Then, since $W$ is admissible,  we have $b\in W$ or $c\in W$.

First suppose that $c\in W$.
If $W=\{a,c\}$, then $P_{\emptyset}(\Pc)$ is a proper subset of $P_W(\Pc)$,
and so $P_W(\Pc)$ is not a minimal prime ideal of $I(\Cc)$.
Hence, we may assume that $\{a,c\}$ is a proper subset of $W$.
In case of $d \in W$,
it follows that $V=W\setminus \{a\}$
is an admissible set with $\Gc(V)=\Gc(W)$. In case of $b\in W$ it follows that $V=W\setminus \{c\}$
is an admissible set with $\Gc(V)=\Gc(W)$.
Hence in both cases it follows from Theorem~\ref{minimalprimes} that $P_W(\Pc)$ is not a minimal prime ideal of $I(\Cc)$.
On the other hand, in case of $d \notin W$ and $b \notin W$,
it follows that $V=W\setminus \{a,b\}$ is an admissible set with either $\Gc(V)=\Gc(W)$ or $\Gc(W)=\Gc(V) \cup \{ad-bc\}$. Hence by Theorem~\ref{minimalprimes}, $P_W(\Pc)$ is not a minimal prime ideal of $I(\Cc)$.

In the second case, suppose that $c\notin W$. Then $b\in W$. Let $a=(i,j)$ and $p=(k,j)$ with $k>i+1$, and let $[a,p]=\{(l,j)\:\, i\leq l\leq k\}$. Then $b\in [a,c]$. If $W=[a,p]$, then $P_W(\Pc)$ is not a minimal prime ideal of $I(\Cc)$, because in that case $P_{\emptyset}(\Pc)$ is a proper subset of $P_W(\Pc)$. On the other hand, if $W$ is a proper subset of $[a,p]$, then $W$ is not admissible. Hence there exists $e\in [a,p]$ such that $[a,e]\subset W$ and moreover $f$, as indicated in Figure~\ref{smallest}, belongs to  $W$.  We may assume that  $[b, f']\sect W=\emptyset$.  Let $V = W \setminus [a, e']$.  Then $V$ is admissible.  Since $\Gc(V) \setminus \Gc(W)$ consists of those adjacent $2$-minors which are indicated  in Figure~\ref{smallest} as the dark shadowed area, it follows from Theorem~\ref{minimalprimes} that $P_W(\Pc)$ is not a minimal prime ideal of $I(\Cc)$.
\end{proof}

\begin{figure}[hbt]
\begin{center}
\psset{unit=0.7cm}
\begin{pspicture}(4.5,-1)(4.5,4)
\pspolygon[style=fyp,fillcolor=light](6,0)(6,1)(7,1)(7,0)
\pspolygon[style=fyp,fillcolor=light](7,0)(7,1)(8,1)(8,0)
\pspolygon[style=fyp,fillcolor=light](8,0)(8,1)(9,1)(9,0)
\pspolygon[style=fyp,fillcolor=light](5,1)(5,2)(6,2)(6,1)
\pspolygon[style=fyp,fillcolor=light](6,1)(6,2)(7,2)(7,1)
\pspolygon[style=fyp,fillcolor=light](4,1)(4,2)(5,2)(5,1)
\pspolygon[style=fyp,fillcolor=light](8,1)(8,2)(9,2)(9,1)
\pspolygon[style=fyp,fillcolor=dark](3,1)(3,2)(4,2)(4,1)
\pspolygon[style=fyp,fillcolor=dark](2,1)(2,2)(3,2)(3,1)
\pspolygon[style=fyp,fillcolor=dark](1,1)(1,2)(2,2)(2,1)

\rput(7.2,2.2){$p$}
\rput(0.8,2.2){$a$}
\rput(2,2.3){$b$}
\rput(0.8,0.75){$c$}
\rput(2,0.75){$d$}
\rput(4,2.3){$e'$}
\rput(5,2.2){$e$}
\rput(4,0.68){$f'$}
\rput(5,0.68){$f$}
\rput(4.5,-0.5){$\Cc$}
\end{pspicture}
\end{center}
\caption{}\label{smallest}
\end{figure}
\noindent
The vertex $c$ in Figure~\ref{smallest} is the first corner of the path $\Cc$. Therefore, according to Lemma~\ref{nonzerodivisor},  the element $a$ is not contained in any minimal prime ideal of $I(\Cc)$.

\begin{proof}[Proof of Theorem~\ref{radicalpath}]
Let $\Pc=\delta_1,\delta_2,\ldots,\delta_r$ be a path and choose the vertex $a\in \delta_1$ as described in Proposition~\ref{nonzerodivisor}. Then our hypothesis implies that $a$ is a nonzero divisor modulo $I(\Pc)$. The graded version of Lemma 4.4.9 in \cite{BH} implies then that $I(\Pc)$ is a radical ideal if and only if $(a,I(\Pc))$ is a radical ideal. Thus it suffices to show that $(a,I(\Pc))$ is a radical ideal.

Let $\delta_1=ad-bc$ and $\Pc'=\delta_2,\ldots,\delta_r$ be the path which is obtained from $\Pc$ by removing $\delta_1$. Then $(a,I(\Pc))=(a,bc, I(\Pc'))$. Thus $(a,I(\Pc))$ is a radical ideal if $(bc, I(\Pc'))$ is a radical ideal, because $a$ is a variable which does not appear in  $(bc, I(\Pc'))$. Since $c$ is regular modulo $(b,I(\Pc'))$, we may apply Lemma~\ref{easy} and get that $(bc, I(\Pc'))=(b,I(\Pc'))\sect (c,I(\Pc'))$. By using induction of the length of the path we may assume that $I(\Pc')$ is a radical ideal. Since $c$ does not appear  $I(\Pc')$ it follows that $(c,I(\Pc))$ is a radical ideal. Thus it remains to be shown that $(b,I(\Pc'))$ is a radical ideal. Observe that $b$ is one of the vertices of $\delta_2$. If it is a free vertex we can argue as before. So we may assume that $b$ is not free. The following Figure~\ref{figure1}(i) and Figure~\ref{figure1}(ii) indicate (up to rotation and reflection)  the possible positions of $b$ in $\Pc'$.

\begin{figure}[hbt]
\begin{center}
\psset{unit=0.7cm}
\begin{pspicture}(-0.5,0)(-0.5,4)
\pspolygon[style=fyp,fillcolor=light](-6,1)(-5,1)(-5,2)(-6,2)
\pspolygon[style=fyp,fillcolor=light](-6,2)(-5,2)(-5,3)(-6,3)
\pspolygon[style=fyp,fillcolor=light](-6,3)(-5,3)(-5,4)(-6,4)
\rput(-5.5,0.3){$(i)$}
\rput(-5.5,1.5){$\delta_2$}
\rput(-5.5,2.5){$\delta_3$}
\rput(-5.5,3.5){$\delta_4$}
\rput(-6.3,1){$d$}
\rput(-6.3,2){$b$}
\rput(-6.3,3){$g$}
\rput(-4.7,1){$e$}
\rput(-4.7,2){$f$}
\rput(-4.7,3){$h$}
\rput(-5.5,4.6){$\vdots$}
\rput(-4.3,3.5){$\cdots$}

\pspolygon[style=fyp,fillcolor=light](3,1)(4,1)(4,2)(3,2)
\pspolygon[style=fyp,fillcolor=light](3,2)(4,2)(4,3)(3,3)
\pspolygon[style=fyp,fillcolor=light](4,2)(5,2)(5,3)(4,3)
\rput(3.5,0.3){$(ii)$}
\rput(3.5,1.5){$\delta_2$}
\rput(3.5,2.5){$\delta_3$}
\rput(4.5,2.5){$\delta_4$}
\rput(2.7,1){$d$}
\rput(2.7,2){$b$}
\rput(2.7,3){$g$}
\rput(4.3,1){$e$}
\rput(4.3,1.7){$f$}
\rput(4,3.3){$h$}
\rput(4.5,3.6){$\vdots$}
\rput(5.6,2.5){$\cdots$}
\end{pspicture}
\end{center}
\caption{}\label{figure1}
\end{figure}

In the case of Figure~\ref{figure1}(i) we have $(b,I(\Pc'))=(b,df,gf,I(\Pc''))$ where $\Pc''=\delta_4,\ldots,\delta_r$. Since the variable $b$ does not appear in $(df,gf,I(\Pc''))$ it follows that $(b,I(\Pc'))$ is a radical ideal if and only if $(df,gf,I(\Pc''))$ is a radical ideal.  Applying Lemma~\ref{easy} we see that $(df,gf,I(\Pc''))=(d,gf,I(\Pc''))\sect (f,I(\Pc''))$. By induction hypothesis we may assume that $(f,I(\Pc''))$ is a radical ideal. Thus it remains to be shown that $(d,gf,I(\Pc''))$ is a radical ideal which is the case if $(gf,I(\Pc''))$ is a radical ideal. Once again we apply Lemma~\ref{easy} and get $(gf,I(\Pc''))=(g,I(\Pc''))\sect (f,I(\Pc''))$. By assumption of induction we deduce as before that both ideals $(g,I(\Pc''))$ and $(f,I(\Pc''))$ are radical ideals. Therefore, $(gf,I(\Pc''))$ is a radical ideal.

In the case of Figure~\ref{figure1}(ii) a similar argument works.
\end{proof}

\section{Special contingency tables}

We call a contingency table $T=(a_{ij})$ with support in $\Sc$ {\em special}, if $\Sc=\{(i,j)\:\; x_{ij}\in V(\Cc)\}$ where $\Cc$ is a special configuration.  If $\Sc'\subset \Sc$ we denote by $T_\Sc'$ the restriction of $T$ to $\Sc'$. In other words,
$T_{\Sc'}=(a_{ij})_{(i,j)\in \Sc'}$.

Figure~\ref{contingency} shows the contingency table corresponding to the configuration  $\Cc$ of $2$-adjacent minors  as shown in Figure~\ref{admissible}.
One has to observe that in the corresponding contingency table the entries are displayed in cells whose center coordinates correspond to the vertices of $\Cc$.

\noindent
In this section we want to discuss connectedness of special contingency tables with respect to adjacent moves. For that we shall need the following result.

\begin{Proposition}
\label{latticeideal}
Let $\Cc$ be a special configuration of $2$-adjacent minors. We write $\Cc=\Union_{k=1}^r\Cc_k$ as a disjoint union of connected special configurations of $2$-adjacent minors, and set  $\Sc=\{(i,j)\:\; x_{ij}\in V(\Cc)\}$ and $\Sc_k=\{(i,j)\:\; x_{ij}\in V(\Cc_k)\}$ for $k=1,\ldots,r$. Then $I(\Cc)$ is a lattice basis ideal for the saturated  lattice $\Lc$ consisting of all tables $T=\{a_{ij}\:\; (i,j)\in\Sc\}$ such that for $k=1,\ldots,r$ the tables $T_{\Sc_k}$  have   row  and column sums equal to zero. Moreover, the lattice ideal $I_\Lc$ of $\Lc$ is generated by all inner $2$-minors of $\Cc$, that is, $I_\Lc=P_\emptyset(\Cc)=(P_\emptyset(\Cc_1),\ldots,P_\emptyset(\Cc_r))$.
\end{Proposition}

\begin{proof}
We observed already in the proof of Lemma~\ref{regular} that for any configuration $\Cc$ of adjacent $2$-minors the ideal  $I(\Cc)$ is a lattice basis ideal of a saturated lattice $\Lc$. The basis elements of the lattice which correspond to the adjacent $2$-minors in $\Cc$ are tables of the form  $D=(d_{ij})$ such  that there exists an integer $k\in\{1,\ldots,r\}$  and integers   $i'$ and $j'$ such that $(i',j'), (i',j'+1),(i'+1,j'),(i'+1,j'+1)\in \Sc_k$,  $d_{i',j'}=d_{i'+1,j'+1}=1$, $d_{i',j'+1}=d_{i'+1,j'}=-1$,  and all for all other pairs $(i,j)$ one has  $d_{i,j}=0$. Since any other element in $\Lc$ is $\ZZ$-linear combination of such tables, it follows $\Lc$ consists of all tables $T=\{a_{ij}\:\; (i,j)\in\Sc\}$ as described in the proposition. As noted in Section~1, we have that $I_\Lc=I(\Cc)\: \prod_{(i,j)\in\Sc}x_{ij}$. Thus Theorem~\ref{primeintersection} implies that $I_\Lc=P_\emptyset(\Cc)$ which is exactly the ideal of inner $2$-minors of $\Cc$.
\end{proof}

\begin{figure}[hbt]
\begin{center}
\psset{unit=0.7cm}
\begin{pspicture}(-2,1)(0,5)
\pspolygon[style=fyp,fillcolor=light](-1,1)(0,1)(0,2)(-1,2)
\pspolygon[style=fyp,fillcolor=light](-2,1)(-1,1)(-1,2)(-2,2)
\pspolygon[style=fyp,fillcolor=light](-1,2)(0,2)(0,3)(-1,3)
\pspolygon[style=fyp,fillcolor=light](-2,2)(-1,2)(-1,3)(-2,3)
\pspolygon[style=fyp,fillcolor=light](-2,3)(-1,3)(-1,4)(-2,4)
\pspolygon[style=fyp,fillcolor=light](-1,3)(0,3)(0,4)(-1,4)
\pspolygon[style=fyp,fillcolor=light](0,3)(1,3)(1,4)(0,4)
\pspolygon[style=fyp,fillcolor=light](0,2)(1,2)(1,3)(0,3)
\pspolygon[style=fyp,fillcolor=light](-3,1)(-2,1)(-2,2)(-3,2)
\pspolygon[style=fyp,fillcolor=light](-3,2)(-2,2)(-2,3)(-3,3)
\rput(-2.5,2.5){$d$}
\rput(-0.5,1.5){$j$}
\rput(-1.5,1.5){$i$}
\rput(-0.5,2.5){$f$}
\rput(-1.5,2.5){$e$}
\rput(-1.5,3.5){$a$}
\rput(0.5,3.5){$c$}
\rput(0.5,2.5){$g$}
\rput(-2.5,1.5){$h$}
\rput(-0.5,3.5){$b$}
\end{pspicture}
\end{center}
\caption{}\label{contingency}
\end{figure}

Let $\Cc$ be a special configuration. Recall from Section 3 that all the minimal prime ideals of $I(\Cc)$ are of the form  $P_W(\Cc)=(W, P_\emptyset(\Cc'))$ where $W\subset V(\Cc)$ is an admissible set and  $\Cc'=\{\delta\in \Cc\:\, V(\delta)\sect W=\emptyset\}$.

Let $\Sc=\{(i,j)\:\; x_{ij}\in V(\Cc)$ and $T=\{c_{ij}\:\; (i,j)\in\Sc\}$ be a table with $c_{ij}\in\ZZ$. We define the support of $T$ to be the set $\supp T=\{x_{ij} \in\Sc\:\; c_{ij}\neq 0\}$. Now we have

\begin{Corollary}
\label{inside}
Let $\Cc$ be a special configuration of $2$-adjacent minors
with $\Sc=V(\Cc)$, and let $W\subset \Sc$ be an admissible set of $\Cc$. Then the binomial $f=\prod_{(i,j)\in\Sc} x_{ij}^{a_{ij}}-\prod_{(i,j)\in\Sc} x_{ij}^{b_{ij}}$ with $T=\{a_{ij}\:\; (i,j)\in\Sc\}$ and $T'=\{b_{ij}\:\; (i,j)\in\Sc\}$  belongs to $P_W(\Cc)=(W,P_\emptyset(\Cc'))$ if and only if one of the following conditions $(i)$ and $(ii)$ is satisfied:
\begin{enumerate}
\item[{\em (i)}] the inequalities
\[
\sum_{x_{ij}\in W}a_{ij}\geq 1\quad\text{and}\quad \sum_{x_{ij}\in W}b_{ij}\geq 1
\]
hold;
\item[{\em (ii)}] one has
$\supp(T-T')\subset V(\Cc')$ and the tables $T_{\Sc_K}$ and $T'_{\Sc_k}$ have the same row and column sums for all $k$, where $\Cc_1,\ldots,\Cc_r$ are the  connected special configurations such that $\Cc'=\Union_{k=1}^r\Cc_k$ is the (unique)  disjoint union of the $\Cc_k$, and $\Sc_k=\{(i,j)\:\; x_{ij}\in V(\Cc_k)\}$ for all $k$.
\end{enumerate}
\end{Corollary}

\begin{proof}
 Suppose that  $f\in P_W(\Cc)$ and that $\sum_{x_{ij}\in W}a_{ij}=0$ or $\sum_{x_{ij}\in W}b_{ij}=0$.
 Say,  $\sum_{x_{ij}\in W}a_{ij}=0$. Then $\sum_{x_{ij}\in W}b_{ij}=0$, because otherwise it would follow that $\prod_{(i,j)\in\Sc} x_{ij}^{a_{ij}} \in P_\emptyset(\Cc')S$. Since $P_\emptyset(\Cc')S$ is a prime ideal, this would imply that  one of the variables  $x_{ij}\in V(\Cc)$ would belong to  $P_\emptyset(\Cc')S$, a contradiction. Therefore, $f \in P_\emptyset(\Cc')S$. Let  $S'\subset S$ be the polynomial ring in the variables $x_{ij}\in V(\Cc')$. Then $P_\emptyset(\Cc')$ is an ideal of $S'$. Let $f=ug$, where $u$ is a monomial and $g$ is a binomial whose monomial terms have greatest common divisor $1$, say, $g=\prod_{(i,j)\in\Sc} x_{ij}^{c_{ij}}-\prod_{(i,j)\in\Sc} x_{ij}^{d_{ij}}$ with $T_0=\{c_{ij}\:\; (i,j)\in\Sc\}$ and $T_0'=\{d_{ij}\:\; (i,j)\in\Sc\}$. Since $f \in P_\emptyset(\Cc')S$ it follows that $g\in P_\emptyset(\Cc')$. Therefore, since $T-T'=T_0-T_0'$, we get that
 $\supp(T-T')=\supp(T_0-T_0')\subset V(\Cc')$. Hence  Proposition~\ref{latticeideal} implies that $(T_0)_{\Sc_k}$ and  $(T_0')_{\Sc_k}$ have the same row and column sums. Since $T=T_0+E$ and $T'=T_0'+E$ where $E$ is the table given by the exponents of the monomial $u$, we conclude that  $T_{\Sc_k}$ and  $T'_{\Sc_k}$ have the same row and column sums, too.

 Conversely suppose that $\sum_{(i,j)\in W}a_{ij}\geq 1\quad\text{and}\quad \sum_{(i,j)\in W}b_{ij}\geq 1$. Then $f\in (W)$, and hence $f\in P_W(\Cc)$. On the other hand, if $\supp(T-T')\subset V(\Cc')$ and the tables $T_{\Sc_k}$ and $T'_{\Sc_k}$ have the same row and column sums for all $k$, then, with the notation introduced above, we have that $\supp(T_0-T_0')\subset V(\Cc')$ and the tables $(T_0)_{\Sc_k}$ and $(T_0')_{\Sc_k}$ have the same row and column sums for all $k$. Hence  Proposition~\ref{latticeideal} yields that  $g\in P_\emptyset(\Cc')$,  and this implies that $f\in P_W(\Cc)$.
\end{proof}

\begin{Corollary}
\label{final}
Let $T=\{a_{ij}\:\; (i,j)\in\Sc\}$ and $T'=\{b_{ij}\:\; (i,j)\in\Sc\}$ be two special contingency tables supported in $\Sc$.

{\em (a)} If $T$ and $T'$ are connected with respect to adjacent moves, then, for $T$ and $T'$ either (i) or (ii)  of Corollary~\ref{inside} is satisfied for all  admissible sets $W\subset \Sc$.

{\em (b)} Let $\Cc$ be the configuration of $2$-adjacent minors with $\Sc=\{\{(i,j)\:\; x_{ij}\in V(\Cc)\}$, and assume that $I(\Cc)$ is a radical ideal.   Then  $T$ and $T'$ are connected with respect to adjacent moves, if and only if,  for $T$ and $T'$ either (i) or (ii)  of Corollary~\ref{inside} is satisfied for all  admissible sets $W\subset \Sc$.

In the above statements it is enough to consider admissible sets which correspond to minimal prime ideals of $I(\Cc)$ as characterized in Theorem~\ref{minimalprimes}(b).
\end{Corollary}

\begin{proof}
By \cite[Theorem 1.1]{DES} $T$ and $T'$ are  connected with respect to adjacent moves if and only if the monomial  $f=\prod_{(i,j)\in\Sc} x_{ij}^{a_{ij}}-\prod_{(i,j)\in\Sc} x_{ij}^{b_{ij}}$ belongs to $I(\Cc)$, where $\Cc$ is the configuration of adjacent $2$-minors with $\Sc=\{(i,j)\:\; x_{ij}\in V(\Cc)\}$. Since $\Cc$ is special it follows from Theorem~\ref{primeintersection}  that $f\in P_W(\Cc)$ for all admissible sets $W\in V(\Cc)$. Thus Corollary~\ref{inside} implies statement (a). In the case that $I(\Cc)$ is a radical ideal,  $f\in  I(\Cc)$ if and only if $f\in P_W(\Cc)$ for all admissible sets $W\in V(\Cc)$. This proves (b).
\end{proof}

It follows from Corollary~\ref{final} and Figure~\ref{minimalp} that two contingency tables $T$ and $T'$ as in Figure~\ref{contingency}  with the same row and column sums are connected via adjacent moves if and only the following conditions are satisfied: either both  tables  satisfies  the following  inequalities
\begin{eqnarray*}
a+e+i&\geq & 1\\
b+f+j&\geq & 1\\
e+f+g+i&\geq & 1\\
b+e+f+i&\geq & 1\\
d+e+f+g&\geq & 1\\
b+d+e+f&\geq & 1.
\end{eqnarray*}
or else  $a+e+i=0$ or $b+f+j=0$   for one of the tables and the remaining 4 inequalities hold for both tables. In this case, if  $a+e+i=0$ for one table, then one should have   that $\supp(T-T')\subset\{b,c,f,g\}$,  and if $b+f+j=0$ for one table,   then one should have  that $\supp(T-T')\subset\{d,e,h,i\}$.

{}

\end{document}